\documentclass[11pt,reqno]{amsart}
\usepackage{amsmath,amssymb,amsthm,mathtools}
\usepackage[letterpaper,margin=1.1in]{geometry}
\usepackage[colorlinks=true,linkcolor=blue,citecolor=blue,urlcolor=blue]{hyperref}
\newtheorem{theorem}{Theorem}[section]
\newtheorem{lemma}[theorem]{Lemma}
\newtheorem{proposition}[theorem]{Proposition}
\newtheorem{corollary}[theorem]{Corollary}
\theoremstyle{definition}
\newtheorem{example}[theorem]{Example}
\theoremstyle{remark}
\newtheorem{remark}[theorem]{Remark}
\numberwithin{equation}{section}

\DeclareMathOperator{\Daw}{Daw}
\DeclareMathOperator{\re}{Re}
\newcommand{\eps}{\varepsilon}
\newcommand{\e}{\mathbf{e}}
\newcommand{\ls}{\lambda^{*}}
\newcommand{\B}{\mathcal{B}}
\newcommand{\Z}{\mathbb{Z}}
\newcommand{\Q}{\mathbb{Q}}
\newcommand{\R}{\mathbb{R}}
\newcommand{\C}{\mathbb{C}}
\newcommand{\podf}{\operatorname{pod}}
\newcommand{\stir}[2]{\genfrac{\{}{\}}{0pt}{}{#1}{#2}}
\newcommand{\doi}[1]{\href{https://doi.org/#1}{\mbox{doi:\nolinkurl{#1}}}}
\newcommand{\arxiv}[1]{\href{https://arxiv.org/abs/#1}{\mbox{arXiv:\nolinkurl{#1}}}}
\hypersetup{pdftitle={Radial Limits and Euler Exponents of a False Theta Series with Character Modulo 8},
  pdfauthor={K. Srinivasa Raghava}}

\title[Radial Limits and Euler Exponents of a False Theta Series]{Radial Limits and Euler Exponents of a False Theta Series with Character Modulo 8}

\author{K. Srinivasa Raghava}
\address{Pie Mathematics Association}
\email{srinivasaraghavak@gmail.com}

\begin{document}

\begin{abstract}
Let $T_k=k(k+1)/2$ and let $N_n(q)$ be the series with terms
$(-1)^{T_k}(2k+1)^{2n+1}q^{T_k}$ for $k\ge0$. We study the false theta
function $D=N_0$ and the moment quotients $N_n/D$. We prove an exact
transformation formula for twists of even order. It expresses each moment as a convergent
series of derivatives of the Dawson integral, with the dual twist and its
phase determined by a quadratic Gauss sum. The formula gives the finite
radial values of the moments at roots of unity as nonzero cyclotomic integers
and yields congruences between them. Inside the unit disc we consider the
quotient of $D$ by the theta series with terms $(-1)^{T_k}q^{T_k}$.
Each nonconstant coefficient of this quotient is minus twice a positive
integer that counts divisors and restricted partitions. These counts give
a Lyndon-word formula for the exponents in the formal Euler product of $D$.
The factorization also identifies the unique real zero in $(-1,1)$.
A comparison with a cubic polynomial shows that this simple zero is the
only zero in $|q|<1/2$.
\end{abstract}

\maketitle

\section{Introduction}

Let $T_k=k(k+1)/2$. Jacobi's identity
\begin{equation}\label{jacobi}
1-3q+5q^3-7q^6+9q^{10}-\cdots=\sum_{k\ge0}(-1)^k(2k+1)q^{T_k}=\prod_{n\ge1}(1-q^n)^3
\end{equation}
gives, on replacing $q$ by $-q$, a product for the series $1+3q-5q^3-7q^6+9q^{10}+\cdots$. In
terms of the odd number $m=2k+1$, the signs in these two series are the values at $m$ of the two
odd Dirichlet characters modulo $8$. This paper is about the two even characters modulo $8$, that
is, about the sign patterns $(+,-,-,+)$ and $(+,+,+,+)$, which are again exchanged
by $q\mapsto-q$.

Put $\eps_k=(-1)^{T_k}$, so that $(\eps_0,\eps_1,\eps_2,\dots)=(+,-,-,+,\dots)$ with period four, and for
$|q|<1$ and $n\ge0$ define
\begin{gather*}
D(q)=\sum_{k\ge0}\eps_k(2k+1)q^{T_k}=1-3q-5q^3+7q^6+9q^{10}-11q^{15}-\cdots,\\
N_n(q)=\sum_{k\ge0}\eps_k(2k+1)^{2n+1}q^{T_k},\qquad W_n(q)=\frac{N_n(q)}{D(q)} .
\end{gather*}
Thus $N_0=D$ and $W_0=1$, and $D(-q)=1+3q+5q^3+7q^6+\cdots$ has the sign pattern $(+,+,+,+)$. The
sign $\eps_k$ is the Kronecker symbol $\bigl(\frac{2}{2k+1}\bigr)$, so
$q^{1/8}D(q)=\sum_{m\ge1}\bigl(\frac2m\bigr)\,m\,q^{m^2/8}$. For an odd character the analogous sum
over all integers $m$ is a theta series of weight $3/2$. For an even character the sum over all
integers vanishes, and the sum over $m\ge1$ is a false theta function, a
notion that goes back to Rogers \cite{Rog}; such series are also called
partial theta functions or partial theta series \cite{And81,AB2,GO}. Neither $D(q)$ nor $D(-q)$ is a
modular form (Remark~\ref{nonmodular}). The corresponding series
without the factor $2k+1$ is a theta function,
\begin{equation}\label{psi}
\psi(-q)=\sum_{k\ge0}\eps_kq^{T_k}=1-q-q^3+q^6+q^{10}-\cdots=\prod_{n\ge1}(1-q^{2n-1})(1-q^{4n}),
\end{equation}
where $\psi(q)=\sum_{k\ge0}q^{T_k}$ is Ramanujan's function. The product is a classical identity
of Gauss, and we recover it in Section~\ref{sec:disc}.

The quotients $W_n$ are modelled on a page of Ramanujan's lost notebook \cite[p.~369]{Ram}. There
Ramanujan forms the same quotients from Jacobi's series \eqref{jacobi} and asserts that they are
polynomials in the classical Eisenstein series of weights $2$, $4$ and $6$. This was proved in
\cite{BCLY,BY}, and explicit formulas for all of them were obtained recently by Amdeberhan, Ono and
Singh \cite{AOS}. Nothing of this kind can hold for
the sign pattern $(+,-,-,+)$, because $W_1$ has a pole inside the unit disc
(Corollary~\ref{growth}). What remains is a pair of differential relations. Since $(2k+1)^2=1+8T_k$, for
the operator $\Theta=8q\,\frac{d}{dq}$ we have
\begin{equation}\label{theta-op}
N_n=(1+\Theta)^nD,\qquad W_{n+1}=\Theta W_n+W_1W_n .
\end{equation}

False theta functions, and more generally partial theta series with periodic coefficients, have
been studied intensively because of their role in quantum topology and in the theory of quantum
modular forms \cite{BM,BN,CM,GO,HLSS,LZ,Mur,Zag01,Zag10}. The series
$q^{1/8}D(-q)=\sum m\,q^{m^2/8}$, summed over the odd integers $m\ge1$, is obtained from
$q^{1/8}\psi(q)=\sum q^{m^2/8}$ by multiplying the term with index $m$ by $m=(m^2)^{1/2}$, so it
is formally a derivative of $\psi$ of order one half. For the Dedekind eta
function the analogous half-derivative occurs in Zagier's strange identity \cite{Zag01}, which
links it with Kontsevich's series $\sum_{n\ge0}(q;q)_n$. In the same paper Zagier proved an
identity for the sum of the tails of the eta product. Andrews, Jim\'enez-Urroz and Ono
\cite{AJO} found two infinite families of such identities, now called sums of tails
identities, and derived from them $q$-hypergeometric generating functions for values of
$L$-functions at negative integers. One of their examples is
\begin{equation}\label{AJO}
-2e^{-t/8}\sum_{n\ge0}\frac{(1-e^{-2t})(1-e^{-4t})\cdots(1-e^{-2nt})}
{(1+e^{-t})(1+e^{-3t})\cdots(1+e^{-(2n+1)t})}
=\sum_{n\ge0}L(\chi_2,-2n-1)\,\frac{(-t/8)^n}{n!},
\end{equation}
an identity of formal power series in $t$, where $\chi_2=\bigl(\frac{2}{\cdot}\bigr)$ is the
Kronecker character of $\Q(\sqrt2)$. The right-hand side is the asymptotic expansion of
$e^{-t/8}D(e^{-t})$ as $t\to0^+$ (Theorem~\ref{radial} and Remark~\ref{Lvalues}), so the limits of
the series $N_n$ at $q=1$ are implicit in \cite{AJO}. Sums of tails identities go back to
Ramanujan's lost notebook \cite{And86}; for later work on such identities, and on
$q$-hypergeometric series attached to partial theta functions and to values of $L$-functions, see
\cite{BFR12,DKS,FPST,Hik,LO,Pat}. Radial limits at roots of unity and the quantum modular forms
that they define are studied in \cite{BFR15,BOPR,Fol,FOR}. In particular Rolen and Schneider
\cite{RS} revisit \cite{AJO}, construct a vector-valued quantum modular form whose components
are $q$-hypergeometric series of the kind that appears on the left of \eqref{AJO}, and obtain analogues
of \eqref{AJO} at roots of unity \cite[Corollary~1.1]{RS}. Goswami and Osburn \cite{GO} proved the
quantum modularity of partial theta series with even or odd periodic coefficients, under a
condition on the support of the coefficients; the series
$q^{1/16}D(q^{1/2})=\sum_{m\ge1}\bigl(\frac2m\bigr)\,m\,q^{m^2/16}$ is a difference of two series
covered by their theorem.

We now describe our results. The first is an exact transformation formula. In its simplest case it
states that for $\alpha,\beta>0$ with $\alpha\beta=\pi^2$,
\begin{align}
\alpha^{1/4}\sum_{k\ge0}\eps_ke^{-\alpha(2k+1)^2/8}
&=\beta^{1/4}\sum_{k\ge0}\eps_ke^{-\beta(2k+1)^2/8},\label{intro-theta}\\
\alpha^{3/4}\sum_{k\ge0}\eps_k(2k+1)e^{-\alpha(2k+1)^2/8}
&=\sqrt{\frac8\pi}\;\beta^{1/4}\sum_{k\ge0}\eps_k\Daw'\Bigl((2k+1)\sqrt{\beta/8}\Bigr),\label{intro-false}
\end{align}
where $\Daw$ is the Dawson integral. The first formula is the classical transformation of the
theta function $\psi(-q)$. The second is its counterpart for $D$. The behaviour of false
theta functions under $\tau\mapsto-1/\tau$ has
been described in several ways: by Eichler integrals with the same asymptotic expansions at
rational points \cite{LZ}, by regularised integrals \cite{CM}, by modular completions \cite{BN},
and by principal value integrals \cite{Mur}. In \eqref{intro-false} the transformed function is a
convergent series of values of $\Daw'$. Theorem~\ref{transf} gives such a formula for every $N_n$,
and with the
signs $\eps_k$ replaced by $\lambda^{T_k}$ for an arbitrary root of unity $\lambda$ of even order;
the right-hand side then contains a second root of unity $\ls$ and a normalised Gauss sum. Since
$\Daw'(x)=-1/(2x^2)+O(x^{-4})$ as $x\to\infty$, letting $\alpha\to0$ in these formulas gives the
behaviour of $N_n(q)$ and $W_n(q)$ when $q$ tends radially to a root of unity $\zeta$
(Theorems~\ref{radial} and \ref{oddorder}). If the order of $\zeta$ is congruent to $2$ modulo
$4$, then $W_n(\zeta e^{-t})\sim n!\,(8/t)^n$ as $t\to0^+$. In all other cases $N_n$ and $W_n$ have
finite nonzero limits. The limits of $N_n$ are cyclotomic integers with an elementary generating
function; up to elementary factors they are values of Dirichlet series at positive even
integers, and they satisfy
congruences of Kummer type (Proposition~\ref{arith}). For example,
\begin{equation}\label{intro-limits}
\lim_{q\to1^-}\frac{1-3^3q-5^3q^3+7^3q^6+\cdots}{1-3q-5q^3+7q^6+\cdots}=-11,\qquad
\lim_{q\to i}W_1(q)=\frac{-199+12i}{5},
\end{equation}
the second limit being radial. At $q=1$ the limit of $W_n$ is $(-1)^nS_{2n+1}$, where
$S_1=1$, $S_3=11$, $S_5=361$, $S_7=24611$, and so on, are the Springer numbers of odd index. In
Theorem~\ref{principal} we show that the normalised coefficients in the expansion $D(e^{-t})\sim-\sum_{n\ge0}\sigma_n(t/8)^n/n!$ are positive integers,
$(\sigma_n)=1,12,384,25728,\dots$, whose generating function is the Stieltjes continued fraction
with partial numerators $3\cdot4,\ 4\cdot5,\ 7\cdot8,\ 8\cdot9,\dots$, and that the Hankel
determinants of the limits of $W_n$ are explicit products of the numbers $16k^2(16k^2-1)$.

The second group of results concerns $D$ inside the unit disc. We prove in
Theorem~\ref{positive} that
\begin{equation}\label{intro-positive}
D(q)=\psi(-q)\Bigl(1-2\sum_{N\ge1}c_Nq^N\Bigr),\qquad
(c_N)_{N\ge1}=1,\,1,\,3,\,4,\,5,\,5,\,8,\,12,\,14,\,14,\,20,\dots,
\end{equation}
where every $c_N$ is a positive integer with a simple description in terms of divisors and
partitions. This positivity has strong consequences (Section~\ref{sec:product}). The exponents
$\kappa_r$ in $D(q)=\prod_{r\ge1}(1-q^r)^{\kappa_r}$ are positive: $\kappa_r$ is the number of Lyndon
words of weight $r$ over a suitable alphabet, increased by $1$ unless $r\equiv2\pmod4$. The
reciprocal $1/D=1+3q+9q^2+32q^3+\cdots$ has positive coefficients, and all nonconstant coefficients of $W_1-1$
are negative. The function $D$ has a simple zero $q_0\in(\frac14,\frac13)$. It is the only real zero
in $(-1,1)$ and the only zero in $|q|\le\frac12$. Its reciprocal $q_0^{-1}$
governs the growth of these coefficients (Corollary~\ref{growth}). The coefficients
of $1/D$ are listed in \cite[A202143]{OEIS} (Remark~\ref{oeis-remark}).

Reciprocals of other false theta functions have been studied recently:
Keith \cite{Kei} obtains congruences, asymptotic bounds and combinatorial identities for their
coefficients, and parity results are proved in \cite{JXY}.

Finally, in Section~\ref{sec:term} we show that the Rogers--Fine identity terminates at roots of
unity (Theorem~\ref{terminating}). As a consequence, if $\zeta$ is a root of unity whose order is
not congruent to $2$ modulo $4$, and $h$ is the order of $-\zeta$, then
\begin{equation}\label{intro-finite}
\lim_{t\to0^+}D(\zeta e^{-t})=-2\sum_{r=0}^{h/2-1}\frac{(\zeta^2;\zeta^2)_r}{(-\zeta;\zeta^2)_{r+1}} .
\end{equation}
The sum on the right is the value at $q=\zeta$ of the series
$\sum_{n\ge0}(q^2;q^2)_n/(-q;q^2)_{n+1}$ from \eqref{AJO}, which diverges for every $q$ in the
open unit disc and terminates at these roots of unity.

The transformation formula follows from Poisson summation, a quadratic Gauss sum
and the Fourier transform of $|y|^{2n+1}e^{-y^2}$. The radial limits then follow
from a partial fraction expansion. The shape of the asymptotic expansions at
roots of unity is classical \cite{LZ,Zag01}, and the limits at $q=1$ occur in
\eqref{AJO}. The evaluation \eqref{intro-finite} is contained, up to normalisation,
in \cite[Corollary~2.1]{RS}. Our terminating identity retains the variable $z$
and gives the generating function \eqref{finite-gen}. The moment representation
and continued fractions for the Springer numbers are known \cite{Dum,Sok20}.
Theorem~\ref{principal} identifies the binomial transform of the odd Springer
numbers with the normalised expansion coefficients of $D$ at $q=1$ and gives
the Hankel determinants of the limiting quotients.

The author's preprint~\cite{RaghavaPreprint} studies the same $N_n$, $D$ and $W_n$.
It establishes the finite Fourier transformation of the twists, the radial
classification at all roots of unity and the signed odd Springer values at $q=1$.
It also proves that $D$ has exactly one zero in $|q|<\frac12$ and that all its
Euler exponents are positive. Here Theorem~\ref{transf} uses the finite Fourier
transformation to obtain an identity for every positive radial parameter in terms
of derivatives of the Dawson integral. The factorization \eqref{intro-positive},
derived from the sum of tails identity for $\psi$ in \cite{AJO}, gives a
combinatorial proof of positivity and identifies the exponents by counting Lyndon
words. We include the preprint's Rouch\'e proof in Theorem~\ref{zero} for use
in the coefficient estimates. The exact coefficient divisibilities and the
$2$-adic interpolation of the quotients are developed in the preprint.

\medskip
\noindent\textbf{Notation.} We write $\e(x)=e^{2\pi ix}$. A root of unity of order $h$ is a
primitive $h$th root of unity. As usual $(a;q)_n=\prod_{j=0}^{n-1}(1-aq^j)$ and
$(a;q)_\infty=\prod_{j\ge0}(1-aq^j)$. The Dawson integral is the odd entire function
\[
\Daw(x)=e^{-x^2}\int_0^xe^{v^2}\,dv,\qquad \Daw'(x)=1-2x\Daw(x).
\]
Limits at a root of unity $\zeta$ are radial: $q=\zeta e^{-t}$ with $t\to0^+$. For a power series
$F$ we write $[q^m]F$ for the coefficient of $q^m$.

\section{A Gauss sum and the Dawson integral}\label{sec:prelim}

\begin{lemma}\label{gauss}
Let $h\ge2$ be an even integer and let $\lambda$ be a root of unity of order $h$. Write
$\lambda=\e(u/h)$ with $u$ prime to $h$, choose $u^*$ with $uu^*\equiv-1\pmod h$, and put
$\ls=\e(u^*/h)$; this root of unity depends only on $\lambda$. Define $\chi_\lambda\colon\Z\to\C$ by $\chi_\lambda(m)=\lambda^{(m^2-1)/8}$ for
odd $m$ and $\chi_\lambda(m)=0$ for even $m$, so that $\chi_\lambda(2k+1)=\lambda^{T_k}$. Then
$\chi_\lambda$ is even and has period $4h$, and for every integer $r$
\begin{equation}\label{gauss-ft}
\sum_{m \bmod 4h}\chi_\lambda(m)\,\e\Bigl(\frac{mr}{4h}\Bigr)=2\sqrt h\;\omega_\lambda\,\chi_{\ls}(r),
\qquad
\omega_\lambda=\frac{1}{\sqrt h}\sum_{k=0}^{h-1}\lambda^{T_k}\cos\frac{(2k+1)\pi}{2h}.
\end{equation}
Moreover $(\ls)^*=\lambda$, $|\omega_\lambda|=1$ and $\omega_\lambda\,\omega_{\ls}=1$.
\end{lemma}

\begin{proof}
Since $h$ is even, $u$ is odd and therefore prime to $8h$. Fix $\bar u$ with
$u\bar u\equiv1\pmod{8h}$. Then $-\bar u\equiv u^*\pmod h$, so $\ls=\e(-\bar u/h)$. For odd $m$
we have $\chi_\lambda(m)=\e\bigl(u(m^2-1)/8h\bigr)$. Replacing $m$ by $m+4h$ changes $u(m^2-1)$ by a
multiple of $8h$, so $\chi_\lambda$ has period $4h$, and it is clearly even.

Let $A(r)$ be the sum in \eqref{gauss-ft}. The congruence
$um^2+2mr\equiv u(m+\bar ur)^2-\bar ur^2\pmod{8h}$ gives
\begin{equation}\label{gauss-square}
A(r)=\e\Bigl(-\frac{u+\bar ur^2}{8h}\Bigr)\sum_{m}\e\Bigl(\frac{u(m+\bar ur)^2}{8h}\Bigr),
\end{equation}
where $m$ runs over the odd residues modulo $4h$. Note that $\e(ua^2/8h)$ depends only on $a$
modulo $4h$, and that $\bar u$ is odd. Suppose first that $r$ is odd. Then $m+\bar ur$ is even,
say $m+\bar ur=2j$, and $j$ runs over all residues modulo $2h$. So the sum in \eqref{gauss-square} is
$Q=\sum_{j \bmod 2h}\e(uj^2/2h)$, which does not depend on $r$, and
\[
A(r)=A(1)\,\e\Bigl(-\frac{\bar u(r^2-1)}{8h}\Bigr)=A(1)\,\chi_{\ls}(r).
\]
Suppose next that $r$ is even. Then $a=m+\bar ur$ runs over the odd residues modulo $4h$. The
substitution $a\mapsto a+2h$ permutes these residues and multiplies $\e(ua^2/8h)$ by
$\e(u(a+h)/2)=-1$, because $u$ and $a+h$ are odd. Hence the sum in \eqref{gauss-square} equals
its negative, and $A(r)=0=\chi_{\ls}(r)$.

It remains to study $A(1)=\e(-(u+\bar u)/8h)\,Q$. In $A(1)$ the terms with $m$ and $4h-m$ have
the same value of $\chi_\lambda$, and $\e(m/4h)+\e(-m/4h)=2\cos(\pi m/2h)$. Taking $m=2k+1$ with
$0\le k<h$ we get $A(1)=2\sqrt h\,\omega_\lambda$, which proves \eqref{gauss-ft}. Next,
$4h\,|\omega_\lambda|^2=|Q|^2$, and writing $j=j'+d$ in $Q\overline Q$ we find
\[
|Q|^2=\sum_{d \bmod 2h}\e\Bigl(\frac{ud^2}{2h}\Bigr)\sum_{j' \bmod 2h}\e\Bigl(\frac{udj'}{h}\Bigr)
=2h\Bigl(1+\e\Bigl(\frac{uh}{2}\Bigr)\Bigr)=4h,
\]
because the inner sum is $2h$ if $h$ divides $d$ and $0$ otherwise, and $uh/2$ is an integer.
So $|\omega_\lambda|=1$. The relation $uu^*\equiv-1$ is symmetric, hence $(\ls)^*=\lambda$.
Finally, apply the finite Fourier transform on $\Z/4h\Z$ twice to $\chi_\lambda$. By Fourier
inversion the result is $4h\,\chi_\lambda(-r)=4h\,\chi_\lambda(r)$, and by \eqref{gauss-ft},
used for $\lambda$ and for $\ls$, it is
$2\sqrt h\,\omega_\lambda\cdot2\sqrt h\,\omega_{\ls}\,\chi_\lambda(r)$. Taking $r=1$ gives
$\omega_\lambda\omega_{\ls}=1$.
\end{proof}

\begin{example}\label{exgauss}
For $h=2$ we have $\lambda=\ls=-1$ and
$\omega_{-1}=\frac1{\sqrt2}\bigl(\cos\frac\pi4-\cos\frac{3\pi}4\bigr)=1$. For $\lambda=i$ we have
$h=4$, $\ls=-i$ and $\omega_i=\cos\frac\pi8+i\sin\frac\pi8=e^{\pi i/8}$.
\end{example}

\begin{remark}\label{closed}
Let $Q=\sum_{j \bmod 2h}\e(uj^2/2h)$ and $\bar u$ be as in the proof of Lemma~\ref{gauss}, and
take $u>0$. The classical evaluation of quadratic Gauss sums \cite[Chapter~1]{BEW} gives
$Q=(1+i^u)\bigl(\frac{2h}{u}\bigr)\sqrt{2h}$, where $\bigl(\frac{2h}{u}\bigr)$ is the Jacobi symbol.
Since $A(1)=2\sqrt h\,\omega_\lambda=\e(-(u+\bar u)/8h)\,Q$, it follows that
\[
\omega_\lambda=\Bigl(\frac{2h}{u}\Bigr)\,\e\Bigl(\frac{(-1)^{(u-1)/2}}{8}-\frac{u+\bar u}{8h}\Bigr),
\]
so $\omega_\lambda$ is an explicit root of unity. We shall not use this formula.
\end{remark}

\begin{lemma}\label{dawson}
Let $n\ge0$.
\begin{enumerate}
\item For all real $x$,
\[
\Daw^{(2n+1)}(x)=(-1)^n\,2^{2n+1}\int_0^\infty v^{2n+1}e^{-v^2}\cos(2xv)\,dv .
\]
\item There is a constant $K_n$ such that for all $x>0$
\[
\Bigl|\Daw^{(2n+1)}(x)+\frac{(2n+1)!}{2\,x^{2n+2}}\Bigr|\le\frac{K_n}{x^{2n+4}} .
\]
\end{enumerate}
\end{lemma}

\begin{proof}
(1) Let $F(x)=\int_0^\infty e^{-v^2}\sin(2xv)\,dv$. Then $F(0)=0$, and one integration by parts
gives $F'(x)=\int_0^\infty2v\,e^{-v^2}\cos(2xv)\,dv=1-2xF(x)$. The Dawson integral satisfies the
same initial value problem, so $F=\Daw$. Differentiating $2n+1$ times under the integral sign
gives (1).

(2) Let $g(v)=v^{2n+1}e^{-v^2}$ and $I[g]=\int_0^\infty g(v)\cos(yv)\,dv$ with $y=2x>0$. Two
integrations by parts give $I[g]=-g'(0)\,y^{-2}-y^{-2}I[g'']$. We apply this $n+2$ times. Since
$g^{(2j+1)}(0)=0$ for $j<n$ and $g^{(2n+1)}(0)=(2n+1)!$, the result is
\[
I[g]=\frac{(-1)^{n+1}(2n+1)!}{y^{2n+2}}+\frac{(-1)^{n}g^{(2n+3)}(0)}{y^{2n+4}}
+\frac{(-1)^{n}}{y^{2n+4}}\,I\bigl[g^{(2n+4)}\bigr],
\]
and $\bigl|I[g^{(2n+4)}]\bigr|\le\int_0^\infty|g^{(2n+4)}(v)|\,dv<\infty$. By (1),
$\Daw^{(2n+1)}(x)=(-1)^n2^{2n+1}I[g]$, which gives (2).
\end{proof}

\section{Transformation formulas}\label{sec:transf}

Let $\chi\colon\Z\to\C$ have period $M$, and put
$\widehat\chi(r)=\sum_{m \bmod M}\chi(m)\,\e(mr/M)$. Let $f$ be a continuous function on $\R$ with
Fourier transform $\widehat f(\xi)=\int_{\R}f(y)\,\e(-\xi y)\,dy$, and suppose that
$f(y)=O\bigl((1+|y|)^{-2}\bigr)$ and $\widehat f(\xi)=O\bigl((1+|\xi|)^{-2}\bigr)$. The Poisson
summation formula \cite[Chapter~VII, Corollary~2.6]{SW},
applied to the functions $y\mapsto f(a+My)$ for $a=0,1,\dots,M-1$, gives
\begin{equation}\label{poisson}
\sum_{m\in\Z}\chi(m)f(m)=\frac1M\sum_{r\in\Z}\widehat\chi(r)\,\widehat f\Bigl(\frac rM\Bigr).
\end{equation}

\begin{theorem}\label{transf}
Let $\lambda$ be a root of unity of even order $h$, and let $\ls$ and $\omega_\lambda$ be as in
Lemma~\ref{gauss}. If $\alpha,\beta>0$ and $\alpha\beta=4\pi^2/h^2$, then
\begin{equation}\label{tr-theta}
\alpha^{1/4}\sum_{k\ge0}\lambda^{T_k}e^{-\alpha(2k+1)^2/8}
=\omega_\lambda\,\beta^{1/4}\sum_{k\ge0}(\ls)^{T_k}e^{-\beta(2k+1)^2/8},
\end{equation}
and for every $n\ge0$
\begin{multline}\label{tr-false}
\alpha^{n+3/4}\sum_{k\ge0}\lambda^{T_k}(2k+1)^{2n+1}e^{-\alpha(2k+1)^2/8}\\
=(-2)^n\,\omega_\lambda\sqrt{\frac8\pi}\;\beta^{1/4}\sum_{k\ge0}(\ls)^{T_k}\,
\Daw^{(2n+1)}\Bigl((2k+1)\sqrt{\beta/8}\Bigr).
\end{multline}
All these series converge absolutely.
\end{theorem}

\begin{proof}
We use \eqref{poisson} with $\chi=\chi_\lambda$ and $M=4h$. By Lemma~\ref{gauss},
$\widehat{\chi_\lambda}=2\sqrt h\,\omega_\lambda\,\chi_{\ls}$. If $f$ is even, then so is
$\widehat f$, both sides of \eqref{poisson} are twice the corresponding sums over the positive odd
integers, and \eqref{poisson} becomes
\begin{equation}\label{poisson-odd}
\sum_{k\ge0}\lambda^{T_k}f(2k+1)=\frac{\omega_\lambda}{2\sqrt h}\sum_{k\ge0}(\ls)^{T_k}\,
\widehat f\Bigl(\frac{2k+1}{4h}\Bigr).
\end{equation}
Take first $f(y)=e^{-\alpha y^2/8}$. Then
\[
\widehat f(\xi)=\sqrt{\frac{8\pi}{\alpha}}\;e^{-8\pi^2\xi^2/\alpha},\qquad
\widehat f\Bigl(\frac{2k+1}{4h}\Bigr)=\sqrt{\frac{8\pi}{\alpha}}\;e^{-\beta(2k+1)^2/8},
\]
because $\pi^2/(2h^2\alpha)=\beta/8$. Since
$\frac{1}{2\sqrt h}\sqrt{8\pi/\alpha}=(\beta/\alpha)^{1/4}$, this is \eqref{tr-theta}.

Take next $f(y)=|y|^{2n+1}e^{-\alpha y^2/8}$. The substitution $y=v\sqrt{8/\alpha}$ and
Lemma~\ref{dawson}(1) give
\[
\widehat f(\xi)=2\int_0^\infty y^{2n+1}e^{-\alpha y^2/8}\cos(2\pi\xi y)\,dy
=(-1)^n\,\frac8\alpha\Bigl(\frac2\alpha\Bigr)^{n}\,\Daw^{(2n+1)}\Bigl(\pi\xi\sqrt{8/\alpha}\Bigr),
\]
and this is $O\bigl((1+|\xi|)^{-2n-2}\bigr)$ by Lemma~\ref{dawson}(2). In particular the series
on the right of \eqref{tr-false} converges absolutely. Since
$\pi\sqrt{8/\alpha}/(4h)=\pi/(h\sqrt{2\alpha})=\sqrt{\beta/8}$, formula \eqref{poisson-odd} becomes
\begin{equation}\label{tr-raw}
\sum_{k\ge0}\lambda^{T_k}(2k+1)^{2n+1}e^{-\alpha(2k+1)^2/8}
=(-1)^n\,\frac{2^{n+2}\,\omega_\lambda}{\sqrt h\;\alpha^{n+1}}\sum_{k\ge0}(\ls)^{T_k}\,
\Daw^{(2n+1)}\Bigl(\frac{(2k+1)\pi}{h\sqrt{2\alpha}}\Bigr).
\end{equation}
Multiplying by $\alpha^{n+3/4}$ and using $\alpha^{-1/4}=\sqrt{h/2\pi}\;\beta^{1/4}$ we obtain
\eqref{tr-false}.
\end{proof}

For $h=2$ we have $\lambda=\ls=-1$, $\omega_{-1}=1$ and $\lambda^{T_k}=\eps_k$, so
\eqref{tr-theta} is \eqref{intro-theta}, and \eqref{tr-false} with $n=0$ is
\eqref{intro-false}. With $\alpha=\beta=\pi$ the latter reads
\[
e^{-\pi/8}D(e^{-\pi})=\frac{2\sqrt2}{\pi}\sum_{k\ge0}\eps_k\,\Daw'\Bigl((2k+1)\sqrt{\pi/8}\Bigr).
\]

For $\lambda$ of odd order the sequence $\lambda^{T_k}$ has a nonzero mean value, which gives a main term.

\begin{theorem}\label{transf-odd}
Let $\lambda$ be a root of unity of odd order $h$ and let $G_\lambda=\sum_{k=0}^{h-1}\lambda^{T_k}$.
Then $|G_\lambda|=\sqrt h$, and for every $n\ge0$ and all $\alpha>0$
\begin{equation}\label{odd-main}
\sum_{k\ge0}\lambda^{T_k}(2k+1)^{2n+1}e^{-\alpha(2k+1)^2/8}
=\frac{G_\lambda}{4h}\;n!\,\Bigl(\frac8\alpha\Bigr)^{n+1}+O(1),
\end{equation}
where the implied constant depends only on $n$ and $h$. For $\lambda=1$ there is an exact
formula: if $\alpha\beta=\pi^2$, then
\begin{equation}\label{tr-one}
\sum_{k\ge0}(2k+1)e^{-\alpha(2k+1)^2/8}=\frac2\alpha\sum_{j\in\Z}(-1)^j\,\Daw'\bigl(j\sqrt{2\beta}\bigr).
\end{equation}
\end{theorem}

\begin{proof}
Let $\chi(m)=\lambda^{(m^2-1)/8}$ for odd $m$ and $\chi(m)=0$ for even $m$. Since $h$ is odd,
$T_{k+h}-T_k=kh+h(h+1)/2$ is a multiple of $h$. So $\chi$ is even and has period $2h$, hence also
period $M=4h$, and with this $M$ we have $\widehat\chi(0)=2G_\lambda$ and $|\widehat\chi(r)|\le2h$
for all $r$. Apply \eqref{poisson} to $f(y)=|y|^{2n+1}e^{-\alpha y^2/8}$. The left side is twice
the left side of \eqref{odd-main}. On the right, the term $r=0$ is
$\frac{1}{4h}\,2G_\lambda\,\widehat f(0)$ with $\widehat f(0)=n!\,(8/\alpha)^{n+1}$, which gives the
main term. By Lemma~\ref{dawson}(2) there is a constant $K_n'$ with
$|\Daw^{(2n+1)}(x)|\le K_n'\,|x|^{-2n-2}$ for all real $x\ne0$. With the formula for $\widehat f$
from the proof of Theorem~\ref{transf} this yields
$|\widehat f(\xi)|\le K_n'\,4^{-n}\,(\pi|\xi|)^{-2n-2}$ for $\xi\ne0$, a bound that does not depend
on $\alpha$. Hence the terms with $r\ne0$ contribute $O(1)$.

Choose $\bar2$ with $2\cdot\bar2\equiv1\pmod h$, so that $T_k\equiv\bar2\,(k^2+k)\pmod h$. Writing
$\lambda=\e(u/h)$ and $k=k'+d$ in $G_\lambda\overline{G_\lambda}$ we get
\[
|G_\lambda|^2=\sum_{d \bmod h}\e\Bigl(\frac{u\bar2\,(d^2+d)}{h}\Bigr)\sum_{k' \bmod h}\e\Bigl(\frac{udk'}{h}\Bigr)=h,
\]
since only $d\equiv0$ contributes. Finally, for $\lambda=1$ the Poisson summation formula for
the odd integers reads $\sum_{m\ \mathrm{odd}}f(m)=\frac12\sum_{r\in\Z}(-1)^r\widehat f(r/2)$. For
$n=0$ we have $\widehat f(r/2)=\frac8\alpha\,\Daw'\bigl(\pi r\sqrt{2/\alpha}\bigr)$ and
$\pi\sqrt{2/\alpha}=\sqrt{2\beta}$, which gives \eqref{tr-one}.
\end{proof}

\section{Radial limits}\label{sec:radial}

In this section $\zeta$ is a root of unity, $\lambda=-\zeta$, and $h$ is the order of $\lambda$.
Since $\eps_k\zeta^{T_k}=\lambda^{T_k}$ and $T_k=((2k+1)^2-1)/8$, we have for $t>0$
\begin{equation}\label{Nrad}
N_n(\zeta e^{-t})=e^{t/8}\sum_{k\ge0}\lambda^{T_k}(2k+1)^{2n+1}e^{-t(2k+1)^2/8} .
\end{equation}
The order $h$ is odd exactly when the order of $\zeta$ is congruent to $2$ modulo $4$.

Suppose that $h$ is even. Then $\lambda^{h/2}=-1$, and from $T_{k+h}-T_k=kh+\frac h2(h+1)$ and
$T_{h-1-k}-T_k=\frac h2(h-1-2k)$ we get
\begin{equation}\label{anti}
\lambda^{T_{k+h}}=-\lambda^{T_k},\qquad \lambda^{T_{h-1-k}}=-\lambda^{T_k}.
\end{equation}
We put
\[
\B_\lambda(w)=\frac{1}{\cosh(hw)}\sum_{k=0}^{h/2-1}\lambda^{T_k}\sinh\bigl((h-1-2k)w\bigr).
\]
For example $\B_{-1}(w)=\sinh w/\cosh2w$ and $\B_i(w)=(\sinh3w+i\sinh w)/\cosh4w$.

\begin{lemma}\label{pf}
Let $\lambda$ be a root of unity of even order $h$.
\begin{enumerate}
\item For $\re w>0$,\quad $\displaystyle\sum_{k\ge0}\lambda^{T_k}e^{-(2k+1)w}=\B_\lambda(w)$.
\item For every complex $w$ that is not a pole of $\B_\lambda$,
\[
\B_\lambda(w)=\frac{\omega_\lambda\,w}{\sqrt h}\sum_{k\ge0}\frac{(\ls)^{T_k}}{w^2+\bigl(\frac{(2k+1)\pi}{2h}\bigr)^2}.
\]
\end{enumerate}
\end{lemma}

\begin{proof}
(1) By the first relation in \eqref{anti} the series equals
$\sum_{k<h}\lambda^{T_k}e^{-(2k+1)w}/(1+e^{-2hw})$. By the second relation, the terms with indices
$k$ and $h-1-k$ in the numerator add up to $2e^{-hw}\lambda^{T_k}\sinh\bigl((h-1-2k)w\bigr)$.

(2) The function $\B_\lambda$ is odd and meromorphic, with at most simple poles at the points
$b_m=i\pi m/2h$, $m$ odd. Since
\[
\sinh(hb_m)=i(-1)^{(m-1)/2},\qquad
\sinh\bigl((h-1-2k)b_m\bigr)=i(-1)^{(m-1)/2}\cos\frac{(2k+1)\pi m}{2h},
\]
the residue at $b_m$ is
\[
\rho_m=\frac1h\sum_{k=0}^{h/2-1}\lambda^{T_k}\cos\frac{(2k+1)\pi m}{2h}
=\frac{1}{4h}\sum_{s \bmod 4h}\chi_\lambda(s)\,\e\Bigl(\frac{sm}{4h}\Bigr)
=\frac{\omega_\lambda}{2\sqrt h}\,\chi_{\ls}(m).
\]
Here the second equality follows by pairing $s$ with $4h-s$ and then, by \eqref{anti}, the
index $k$ with $h-1-k$; the third is Lemma~\ref{gauss}.

Let $R_j$ be the rectangle with vertices $\pm j\pm i\pi j/h$, where $j$ is a positive integer. No
pole lies on its boundary, and $|\B_\lambda|\le h$ there. Indeed, for $0<c<h$ we have
$|\sinh(cw)|\le\cosh(c\re w)$; on the horizontal sides $|\cosh(hw)|=\cosh(h\re w)$, and on the
vertical sides $|\cosh(hw)|\ge\sinh(hj)\ge\frac12\cosh((h-1)j)$, the last inequality because
$e^{j}\ge e$. Fix $w$, not a pole, and let $j$ be so large that $w$ lies inside $R_j$. By the
residue theorem
\[
\frac{1}{2\pi i}\oint_{\partial R_j}\frac{\B_\lambda(v)}{v-w}\,dv
=\B_\lambda(w)-\sum_{|m|<2j}\frac{\rho_m}{w-b_m},
\]
the sum being over odd $m$. The left side tends to $0$ as $j\to\infty$. To see this, write
$\frac{1}{v-w}=\frac1v+\frac{w}{v(v-w)}$. The integral of $\B_\lambda(v)/v$ vanishes, because the
integrand is even and $\partial R_j$ is symmetric about the origin; the remaining integrand is
$O(j^{-2})$ on a contour of length $O(j)$. Hence
$\B_\lambda(w)=\lim_{j\to\infty}\sum_{|m|<2j}\rho_m/(w-b_m)$, and combining the terms $m$ and $-m$
gives (2). This is the classical expansion of a meromorphic function in partial fractions
\cite[Section~7.4]{WW}.
\end{proof}

\begin{theorem}\label{radial}
Let $\zeta$ be a root of unity such that $\lambda=-\zeta$ has even order $h$, and let $n\ge0$.
\begin{enumerate}
\item The limit $L_n(\lambda)=\lim_{t\to0^+}N_n(\zeta e^{-t})$ exists, and
\begin{equation}\label{Lseries}
L_n(\lambda)=(-1)^{n+1}\,\frac{\omega_\lambda}{\sqrt h}\,(2n+1)!\,\Bigl(\frac{2h}{\pi}\Bigr)^{2n+2}
\sum_{k\ge0}\frac{(\ls)^{T_k}}{(2k+1)^{2n+2}} .
\end{equation}
In particular $L_n(\lambda)\ne0$.
\item For $|x|<\pi/2h$,
\begin{equation}\label{Lgen}
\sum_{n\ge0}L_n(\lambda)\,\frac{x^{2n+1}}{(2n+1)!}=-\B_\lambda(x)
=-\frac{1}{\cosh(hx)}\sum_{k=0}^{h/2-1}\lambda^{T_k}\sinh\bigl((h-1-2k)x\bigr).
\end{equation}
In particular $L_0(\lambda)=\sum_{k=0}^{h-1}k\,\lambda^{T_k}$.
\item $D(\zeta e^{-t})\ne0$ for all sufficiently small $t>0$, and
$\displaystyle\lim_{t\to0^+}W_n(\zeta e^{-t})=\frac{L_n(\lambda)}{L_0(\lambda)}\ne0$.
\item For every $J\ge1$, as $t\to0^+$,
\[
N_n(\zeta e^{-t})=e^{t/8}\sum_{j=0}^{J-1}\frac{L_{n+j}(\lambda)}{j!}\Bigl(-\frac t8\Bigr)^{j}+O(t^J).
\]
\end{enumerate}
\end{theorem}

\begin{proof}
(1) Put $x_k=(2k+1)\pi/(h\sqrt{2t})$. By \eqref{Nrad} and \eqref{tr-raw} with $\alpha=t$,
\[
e^{-t/8}N_n(\zeta e^{-t})=(-1)^n\,\frac{2^{n+2}\,\omega_\lambda}{\sqrt h\;t^{n+1}}
\sum_{k\ge0}(\ls)^{T_k}\,\Daw^{(2n+1)}(x_k).
\]
By Lemma~\ref{dawson}(2),
$\Daw^{(2n+1)}(x_k)=-\frac12(2n+1)!\,x_k^{-2n-2}+\theta_kK_nx_k^{-2n-4}$ with $|\theta_k|\le1$. Since
$x_k^{-2}=2h^2t\,\pi^{-2}(2k+1)^{-2}$, the main terms add up to the right side of \eqref{Lseries}, and
the remaining terms add up to $O(t)$. In the series \eqref{Lseries} the term $k=0$ is $1$ and the
others have absolute value at most $\sum_{k\ge1}(2k+1)^{-2}=\pi^2/8-1<1$ in total, so
$L_n(\lambda)\ne0$.

(2) For $|x|<\pi/2h$ we expand each term of Lemma~\ref{pf}(2) in powers of $x$. The resulting
double series converges absolutely, and
\[
\B_\lambda(x)=\frac{\omega_\lambda}{\sqrt h}\sum_{n\ge0}(-1)^n\,x^{2n+1}\Bigl(\frac{2h}{\pi}\Bigr)^{2n+2}
\sum_{k\ge0}\frac{(\ls)^{T_k}}{(2k+1)^{2n+2}} .
\]
Comparison with \eqref{Lseries} gives \eqref{Lgen}. The coefficient of $x$ in \eqref{Lgen} is
$L_0(\lambda)=\sum_{k<h/2}(2k+1-h)\lambda^{T_k}$, which equals $\sum_{k<h}k\lambda^{T_k}$ by
\eqref{anti}.

(3) follows from (1), because $D=N_0$ and $L_0(\lambda)\ne0$.

(4) Let $V_n(t)=e^{-t/8}N_n(\zeta e^{-t})$ for $t>0$. Differentiating \eqref{Nrad} term by term
gives $V_n'=-\frac18V_{n+1}$, so $V_n^{(j)}=(-\frac18)^jV_{n+j}$, and by (1) this has the limit
$(-\frac18)^jL_{n+j}(\lambda)$ as $t\to0^+$. Taylor's formula with integral remainder on
$[\delta,t]$, in the limit $\delta\to0^+$, gives
\[
V_n(t)=\sum_{j=0}^{J-1}\frac{L_{n+j}(\lambda)}{j!}\Bigl(-\frac t8\Bigr)^{j}+O(t^J),
\]
because $V_n^{(J)}$ is bounded near $0$.
\end{proof}

\begin{theorem}\label{oddorder}
Let $\zeta$ be a root of unity whose order is congruent to $2$ modulo $4$, so that
$\lambda=-\zeta$ has odd order $h$, and let $G_\lambda$ be as in Theorem~\ref{transf-odd}. Then, as
$t\to0^+$,
\[
N_n(\zeta e^{-t})=e^{t/8}\Bigl\{\frac{G_\lambda}{4h}\;n!\,\Bigl(\frac8t\Bigr)^{n+1}+O(1)\Bigr\},
\qquad \Bigl(\frac t8\Bigr)^{n}W_n(\zeta e^{-t})=n!+O(t).
\]
\end{theorem}

\begin{proof}
The first formula is \eqref{odd-main} with $\alpha=t$, inserted in \eqref{Nrad}. Since
$G_\lambda\ne0$, the case $n=0$ shows that $D(\zeta e^{-t})\ne0$ for small $t$, and dividing the
formula for $N_n$ by the one for $N_0$ gives the second.
\end{proof}

\begin{example}\label{exroots}
(a) Let $\zeta=-i$, so that $\lambda=i$ and $h=4$. By \eqref{Lgen} the numbers $L_n(i)$ are the
coefficients of $-(\sinh3x+i\sinh x)/\cosh4x$. Thus
\[
L_0(i)=-3-i,\qquad L_1(i)=117+47i,\qquad L_2(i)=-15123-6241i,
\]
and $W_1(q)\to\frac{-199-12i}{5}$, $W_2(q)\to5161+360i$ as $q\to-i$. Since the series
$N_n$ have real coefficients, the limits at $\zeta=i$ are the complex conjugates, which gives the
second limit in \eqref{intro-limits}. Formula
\eqref{Lseries} with $n=0$, together with Example~\ref{exgauss}, gives the evaluation
\[
1-\frac{i}{3^2}+\frac{i}{5^2}-\frac1{7^2}-\frac1{9^2}+\frac{i}{11^2}-\cdots
=\sum_{k\ge0}\frac{(-i)^{T_k}}{(2k+1)^2}=\frac{\pi^2\,(3+i)}{32}\,e^{-\pi i/8}.
\]

(b) Let $\zeta=e^{2\pi i/3}$, so that $\lambda=e^{-\pi i/3}$ and $h=6$. Then
$L_0(\lambda)=-4-3\lambda=\frac12(-11+3\sqrt3\,i)$ and
$W_1(q)\to\frac{1}{37}(-3179+132\sqrt3\,i)$ as $q\to e^{2\pi i/3}$.

(c) Let $\zeta=-1$, so that $\lambda=1$, $h=1$ and $G_1=1$. Then $W_n(-e^{-t})\sim n!\,(8/t)^n$,
and from \eqref{tr-one} and $\Daw'(x)=-1/(2x^2)+O(x^{-4})$ one finds
$D(-e^{-t})=\frac2t+\frac13+O(t)$.
\end{example}

\begin{remark}\label{nonmodular}
The radial behaviour shows directly that $D$ is not a modular form. Let $b$ be real and suppose
that $f(\tau)=q^bD(q)$, with $q=\e(\tau)$, were a modular form of some real weight $k$ on a
subgroup of finite index in $\mathrm{SL}_2(\Z)$, with any multiplier system and with poles
allowed at the cusps. By the Fourier expansion at a cusp, $|f(\tau)|$ is then, as $\tau$ tends to
that cusp along a vertical line, either exponentially small, or exponentially large, or of
exact order $(\operatorname{Im}\tau)^{-k}$. But $|f|$ tends to $1$ at the cusp $\tau=0$, where
$q\to1$, and it has exact order $(\operatorname{Im}\tau)^{-1}$ at the cusp $\tau=\frac12$, where
$q\to-1$. These two facts are incompatible. The same argument, with the two cusps interchanged,
applies to $D(-q)$.
\end{remark}

\begin{remark}\label{Lvalues}
The numbers $L_n(\lambda)$ are values of a Dirichlet series at negative integers. Indeed, by
\eqref{anti} the periodic function $\chi_\lambda$ has mean value zero. For such a function the
Dirichlet series $L(s,\chi_\lambda)=\sum_{m\ge1}\chi_\lambda(m)m^{-s}$ continues to an entire
function, and $\sum_{m\ge1}\chi_\lambda(m)e^{-mx}$ has the expansion
$\sum_{j\ge0}L(-j,\chi_\lambda)(-x)^j/j!$ at $x=0$; see \cite[p.~98]{LZ}. Comparing with
Lemma~\ref{pf}(1) and \eqref{Lgen} we find $L_n(\lambda)=L(-2n-1,\chi_\lambda)$. For $\lambda=-1$
the function $\chi_\lambda$ is the character $\chi_2$, and we recover the coefficients in
\eqref{AJO}. In this language \eqref{Lseries} is an instance of the functional equation of
$L(s,\chi_\lambda)$, and parts (2) and (4) of Theorem~\ref{radial} are instances of the general
principle of \cite{LZ,Zag01}, see also \cite{Hik}; for the series on the left of \eqref{AJO},
expansions of this kind at roots of unity are given in \cite[Corollary~1.1]{RS}. The proof through
Theorem~\ref{transf} keeps the argument self-contained and
gives the exact formula \eqref{tr-raw} for every $t>0$, not only the expansion at $t=0$.
\end{remark}

We turn to the arithmetic properties of the limits. Let $E_{2j}$ be the Euler numbers, defined by
$1/\cosh x=\sum_{j\ge0}E_{2j}\,x^{2j}/(2j)!$, so that $E_0=1$, $E_2=-1$, $E_4=5$ and $E_6=-61$. For a
nonzero integer $a$ let $\nu_2(a)$ be the exponent of $2$ in $a$.

\begin{proposition}\label{arith}
Let $\lambda$ be a root of unity of even order $h$, and let $n,n'\ge0$.
\begin{enumerate}
\item We have
\[
L_n(\lambda)=-\sum_{k=0}^{h/2-1}\lambda^{T_k}\sum_{j=0}^{n}\binom{2n+1}{2j}E_{2j}\,h^{2j}\,(h-1-2k)^{2n+1-2j}.
\]
In particular $L_n(\lambda)\in\Z[\lambda]$.
\item $\gamma\bigl(L_n(\lambda)\bigr)=L_n\bigl(\gamma(\lambda)\bigr)$ for every automorphism
$\gamma$ of the field $\Q(\lambda)$.
\item If $p$ is an odd prime and $n\equiv n'\pmod{(p-1)/2}$, then
$L_n(\lambda)\equiv L_{n'}(\lambda)\pmod{p\,\Z[\lambda]}$.
\item If $n\ne n'$, then $L_n(\lambda)\equiv L_{n'}(\lambda)$ modulo $2^{\,2+\nu_2(n-n')}\,\Z[\lambda]$,
and modulo $2^{\,3+\nu_2(n-n')}\,\Z[\lambda]$ if $4\mid h$.
\end{enumerate}
\end{proposition}

\begin{proof}
(1) is obtained by inserting the Taylor expansion of $\sinh(cx)/\cosh(hx)$ in \eqref{Lgen}, and
(2) follows from (1), because the formula in (1) is the same for all roots of unity of order $h$.

(3) By the proof of Lemma~\ref{pf}(1), $\B_\lambda(x)=R(e^{-x})$ identically, with the rational
function $R(y)=\sum_{k<h}\lambda^{T_k}y^{2k+1}/(1+y^{2h})$. Let $\vartheta=y\,\frac{d}{dy}$. Then \eqref{Lgen}
gives $L_n(\lambda)=(\vartheta^{2n+1}R)(1)$. Write $R(y)=\sum_{i\ge0}r_i(y-1)^i$ near $y=1$. Since
$1+y^{2h}=2\bigl(1+\sum_{i\ge1}\frac12\binom{2h}{i}(y-1)^i\bigr)$, all $r_i$ lie in
$\Z[\lambda,\frac12]$. The operator identity
$\vartheta^m=\sum_{i}\stir mi\,y^i\frac{d^i}{dy^i}$, with Stirling numbers of the second kind, gives
\[
L_n(\lambda)=\sum_{i=1}^{2n+1}r_i\;i!\,\stir{2n+1}{i},\qquad
i!\,\stir mi=\sum_{l=0}^{i}(-1)^{i-l}\binom il\,l^m .
\]
If $m,m'\ge1$ and $m\equiv m'\pmod{p-1}$, then $l^m\equiv l^{m'}\pmod p$ for all integers $l$.
With $m=2n+1$ and $m'=2n'+1$ we get
$L_n(\lambda)-L_{n'}(\lambda)\in p\,\Z[\lambda,\frac12]\cap\Z[\lambda]=p\,\Z[\lambda]$, because $p$ is
odd.

(4) We may assume $n>n'$. Put $\nu=\nu_2(n-n')$. Fix an odd integer $c$ and let
$v_n=\sum_{j=0}^{n}\binom{2n+1}{2j}E_{2j}h^{2j}c^{2n+1-2j}$. For $0\le j\le n$ the contribution of
the index $j$ to $v_n-v_{n'}$ is $E_{2j}h^{2j}$ times
\[
\Bigl(\binom{2n+1}{2j}-\binom{2n'+1}{2j}\Bigr)c^{2n+1-2j}
+\binom{2n'+1}{2j}\bigl(c^{2(n-n')}-1\bigr)c^{2n'+1-2j},
\]
where the second term is absent if $j>n'$.
Since $c^2\equiv1\pmod8$, the number $c^{2(n-n')}-1$ is divisible by $2^{3+\nu}$. For the first
term, Vandermonde's identity and $\binom zi=\frac zi\binom{z-1}{i-1}$ give
$\nu_2\bigl(\binom{a+z}{m}-\binom am\bigr)\ge\nu_2(z)-\lfloor\log_2m\rfloor$ for positive integers
$a$, $z$, $m$. With $z=2(n-n')$ and $m=2j$ the first term, multiplied by $h^{2j}$, is divisible by
$2$ to the power $2j\,\nu_2(h)+1+\nu-\lfloor\log_2(2j)\rfloor$ when $j\ge1$, and it vanishes when
$j=0$. This exponent is $\nu+2\nu_2(h)$ for $j=1$ and at least $\nu+3$ for $j\ge2$. Hence
$v_n\equiv v_{n'}$ modulo $2^{\nu+\min(3,\,2\nu_2(h))}$. This exponent is $\nu+2$ if $h\equiv2\pmod4$
and $\nu+3$ if $4\mid h$, and (4) follows from (1).
\end{proof}

\begin{example}\label{exarith}
(a) For $\lambda=-1$ we shall see in Section~\ref{sec:principal} that
$L_n(-1)=(-1)^{n+1}S_{2n+1}$, where $S_{2n+1}=1,11,361,24611,\dots$ are the Springer numbers of odd
index. Part (4) gives $(-1)^nS_{2n+1}\equiv1$ modulo $2^{2+\nu_2(n)}$ for $n\ge1$, and the exponent is best possible:
$\nu_2\bigl((-1)^nS_{2n+1}-1\bigr)=\nu_2(n)+2$. Indeed, in the proof of (4) with $c=1$ and $n'=0$ the
term $j=1$ is $-4n(2n+1)$, and the terms with $j\ge2$ are divisible by $2^{\nu_2(n)+3}$. Part (3) with
$p=3$ gives $(-1)^nS_{2n+1}\equiv1\pmod3$ for all $n$. By Remark~\ref{Lvalues} the numbers
$L_n(-1)=L(-2n-1,\chi_2)$ are, up to elementary factors, generalised Bernoulli numbers, and for
these congruences of Kummer type are classical \cite{Car}. The proof of
Proposition~\ref{arith} is elementary and uniform in $\lambda$.

(b) For $\lambda=i$ we have $L_0(i)=-(1+i)(2-i)$. By (4) the prime $1+i$ divides every $L_n(i)$.
By (3) with $p=5$ the prime $2-i$ divides $L_n(i)$ if and only if $n$ is even, since it divides
$L_0(i)$ and does not divide $L_1(i)=117+47i$. Hence the limit of $W_n(q)$ as $q\to-i$ is a
Gaussian integer exactly when $n$ is even, in agreement with Example~\ref{exroots}(a).
\end{example}

\section{The point \texorpdfstring{$q=1$}{q=1}}\label{sec:principal}

At $\zeta=1$ everything can be made explicit. Let $S_m$ be the Springer numbers, defined by
$1/(\cos x-\sin x)=\sum_{m\ge0}S_m\,x^m/m!$ \cite{Arn,Sok20}. The odd part of this function is
$\sin x/\cos2x$, so that
\begin{equation}\label{springer}
\frac{\sinh x}{\cosh2x}=\sum_{n\ge0}(-1)^nS_{2n+1}\frac{x^{2n+1}}{(2n+1)!},\qquad
(S_{2n+1})_{n\ge0}=1,\,11,\,361,\,24611,\,2873041,\dots .
\end{equation}
These are the coefficients studied by Glaisher \cite{Gla}; see also \cite[A000464]{OEIS}. We also put
\[
\sigma_n=\sum_{j=0}^{n}\binom njS_{2j+1},\qquad (\sigma_n)_{n\ge0}=1,\,12,\,384,\,25728,\,2973696,\dots .
\]

The first formula of the next proposition is due to Sokal \cite[(5.8)]{Sok20}, in an equivalent
form.

\begin{proposition}\label{moments}
Let $d\mu(y)=2\sqrt2\;y\,\dfrac{\sinh(\pi y/2)}{\cosh(\pi y)}\,dy$ on $(0,\infty)$. Then $\mu$ is a
probability measure, and for all $n\ge0$
\[
S_{2n+1}=\int_0^\infty(4y^2)^n\,d\mu(y),\qquad \sigma_n=\int_0^\infty(1+4y^2)^n\,d\mu(y).
\]
\end{proposition}

\begin{proof}
By Lemma~\ref{pf}(1) with $\lambda=-1$ and $w=\pi y/2$ we have, for $y>0$,
\[
\frac{\sinh(\pi y/2)}{\cosh(\pi y)}=\sum_{k\ge0}\eps_k\,e^{-(2k+1)\pi y/2}.
\]
The partial sums of the signs $\eps_k$ take only the values $1,0,-1$, so by partial summation
the partial sums of this series are bounded by $e^{-\pi y/2}$. We may therefore integrate term by term against
$\sin(2xy)$, and Lemma~\ref{pf}(2) with $\omega_{-1}=1$ gives
\[
\int_0^\infty\frac{\sinh(\pi y/2)}{\cosh(\pi y)}\,\sin(2xy)\,dy
=\frac x2\sum_{k\ge0}\frac{\eps_k}{x^2+((2k+1)\pi/4)^2}=\frac{1}{\sqrt2}\,\frac{\sinh x}{\cosh2x}
\]
for real $x$. For $|x|<\pi/4$ we expand $\sin(2xy)$ in powers of $x$, which is permitted because
the integrand decays like $e^{-\pi y/2}$, and compare with \eqref{springer}. This gives the formula for $S_{2n+1}$; the case $n=0$ shows that $\mu$ has total
mass $1$. The formula for $\sigma_n$ follows from the binomial theorem.
\end{proof}

\begin{remark}\label{sine-transform}
The identity in the proof says that the function
$\varphi(x)=\sinh(\sqrt\pi\,x/2)/\cosh(\sqrt\pi\,x)$ is its own Fourier sine transform:
$\varphi(x)=\sqrt{2/\pi}\int_0^\infty\varphi(y)\sin(xy)\,dy$. This reflects the fact that
$\ls=\lambda$ and $\omega_\lambda=1$ for $\lambda=-1$.
\end{remark}

\begin{theorem}\label{principal}
Let $S_m$ and $\sigma_n$ be as above.
\begin{enumerate}
\item For every $n\ge0$,
\[
\lim_{q\to1^-}N_n(q)=(-1)^{n+1}S_{2n+1},\qquad
\lim_{q\to1^-}W_n(q)=(-1)^{n}S_{2n+1}=\sum_{j=0}^{n}\binom{2n+1}{2j}4^jE_{2j},
\]
and
\[
\sum_{k\ge0}\frac{\eps_k}{(2k+1)^{2n+2}}=\frac{\sqrt2\;S_{2n+1}}{(2n+1)!}\Bigl(\frac\pi4\Bigr)^{2n+2}.
\]
\item As $t\to0^+$ we have the asymptotic expansion
$D(e^{-t})\sim-\sum_{n\ge0}\dfrac{\sigma_n}{n!}\Bigl(\dfrac t8\Bigr)^{n}$. Moreover, for every
$n\ge0$,
\[
\lim_{q\to1^-}\frac{\Theta^nD(q)}{D(q)}=(-1)^n\sigma_n,\qquad\text{that is,}\qquad
\lim_{q\to1^-}\frac{\sum_{k\ge0}\eps_k(2k+1)\,T_k^{\,n}\,q^{T_k}}{\sum_{k\ge0}\eps_k(2k+1)\,q^{T_k}}
=\Bigl(-\frac18\Bigr)^{n}\sigma_n .
\]
\item As formal power series in $z$,
\begin{equation}\label{Sfrac}
\sum_{n\ge0}\sigma_nz^n=\cfrac{1}{1-\cfrac{3\cdot4\,z}{1-\cfrac{4\cdot5\,z}{1-\cfrac{7\cdot8\,z}{1-\cfrac{8\cdot9\,z}{1-\cdots}}}}}\;,
\end{equation}
where the partial numerators are $a_{2n+1}=(4n+3)(4n+4)$ and $a_{2n}=4n(4n+1)$. Equivalently,
\begin{equation}\label{Jfrac}
\sum_{n\ge0}S_{2n+1}z^n=\cfrac{1}{1-11z-\cfrac{240\,z^2}{1-75z-\cfrac{4032\,z^2}{1-203z-\cdots}}}\;,
\end{equation}
with coefficients $32n^2+32n+11$ and $16n^2(16n^2-1)$.
\item For every $r\ge1$,
\[
\lim_{q\to1^-}\det\bigl(W_{i+j}(q)\bigr)_{0\le i,j\le r}=\det\bigl(\sigma_{i+j}\bigr)_{0\le i,j\le r}
=\prod_{k=1}^{r}\bigl(16k^2(16k^2-1)\bigr)^{r+1-k}.
\]
\end{enumerate}
\end{theorem}

For example $\sum_{k\ge0}\eps_k(2k+1)^{-4}=1-3^{-4}-5^{-4}+7^{-4}+\cdots=11\sqrt2\,\pi^4/1536$; the
limits in (2) are $-\frac32$, $6$, $-\frac{201}{4}$ for $n=1,2,3$; and $W_2-W_1^2\to240$, while the
$3\times3$ determinant in (4) tends to $240^2\cdot4032$.

\begin{proof}
(1) Here $\lambda=-1$, $h=2$ and $\B_{-1}(x)=\sinh x/\cosh2x$. By \eqref{Lgen} and
\eqref{springer}, $L_n(-1)=(-1)^{n+1}S_{2n+1}$, which gives the two limits by
Theorem~\ref{radial}. The finite sum is Proposition~\ref{arith}(1) with $h=2$, and the series
evaluation is \eqref{Lseries} with $\omega_{-1}=1$.

(2) By Theorem~\ref{radial}(4) with $n=0$ and $L_j(-1)\,(-1)^j=-S_{2j+1}$,
\[
D(e^{-t})\sim-e^{t/8}\sum_{j\ge0}\frac{S_{2j+1}}{j!}\Bigl(\frac t8\Bigr)^{j}
=-\sum_{n\ge0}\frac{1}{n!}\Bigl(\frac t8\Bigr)^{n}\sum_{j=0}^{n}\binom njS_{2j+1},
\]
and the inner sum is $\sigma_n$. By \eqref{theta-op} and (1),
\[
\Theta^nD=\sum_{j=0}^n\binom nj(-1)^{n-j}N_j\ \longrightarrow\
\sum_{j=0}^n\binom nj(-1)^{n-j}(-1)^{j+1}S_{2j+1}=(-1)^{n+1}\sigma_n
\]
as $q\to1^-$, while $D(q)\to-1$. The last formula in (2) is the same statement, since
$\Theta^nD=\sum_k\eps_k(2k+1)(8T_k)^nq^{T_k}$.

(3) Substitute $y=2\tau$ in Proposition~\ref{moments}. The reflection formula for the gamma
function gives $|\Gamma(\frac12+i\tau)|^2=\pi/\cosh\pi\tau$,
$|\Gamma(\frac14+i\tau)\Gamma(\frac34+i\tau)|^2=2\pi^2/\cosh2\pi\tau$ and
$|\Gamma(2i\tau)|^2=\pi/(2\tau\sinh2\pi\tau)$, so that
\[
\biggl|\frac{\Gamma(\frac14+i\tau)\,\Gamma(\frac12+i\tau)\,\Gamma(\frac34+i\tau)}{\Gamma(2i\tau)}\biggr|^2
=8\pi^2\,\frac{\tau\sinh(\pi\tau)}{\cosh(2\pi\tau)} .
\]
Thus $\mu$ is, up to a constant factor, the orthogonality measure of the continuous dual Hahn
polynomials with the positive parameters $(\frac14,\frac12,\frac34)$ in the variable $\tau$
\cite[Section~9.3]{KLS}; these polynomials are limiting cases of the Wilson polynomials \cite{Wil}.
By the normalised recurrence relation in \cite[Section~9.3]{KLS}, the monic orthogonal polynomials $P_n$
in the variable $X=\tau^2+\frac1{16}$ satisfy
\[
XP_n=P_{n+1}+\bigl((n+\tfrac34)(n+1)+n(n+\tfrac14)\bigr)P_n+(n-\tfrac14)\,n\cdot n\,(n+\tfrac14)\,P_{n-1}.
\]
Now $1+4y^2=16X$. Hence the monic orthogonal polynomials for $\mu$ in the variable $1+4y^2$
have the recurrence coefficients $a_{2n}+a_{2n+1}$ and $a_{2n-1}a_{2n}$, with $a_j$ as in the
statement and $a_0=0$. Since $\mu$ has mass $1$ and moments $\sigma_n$, it follows that
$\sum_n\sigma_nz^n$ is the Jacobi continued fraction with these coefficients
\cite[Theorem~13]{Kra}:
\[
\sum_{n\ge0}\sigma_nz^n=\cfrac{1}{1-a_1z-\cfrac{a_1a_2\,z^2}{1-(a_2+a_3)z-\cfrac{a_3a_4\,z^2}{1-(a_4+a_5)z-\cdots}}}\;.
\]
This Jacobi continued fraction equals the Stieltjes continued fraction \eqref{Sfrac}. Indeed, let
$F_j$ be the formal power series defined by the continued fraction \eqref{Sfrac} with partial
numerators $a_j,a_{j+1},\dots$, so
that $F_j=1/(1-a_jzF_{j+1})$. Then
\[
\frac{1}{F_1}=1-\frac{a_1z}{1-a_2zF_3}=1-a_1z-\frac{a_1a_2\,z^2}{F_3^{-1}-a_2z},\qquad
\frac{1}{F_3}-a_2z=1-(a_2+a_3)z-\frac{a_3a_4\,z^2}{F_5^{-1}-a_4z},
\]
and so on. This proves \eqref{Sfrac}. The numbers $S_{2n+1}$ are the moments of
$4y^2=(1+4y^2)-1$. A translation of the variable by $1$ lowers the first recurrence coefficient
by $1$ and leaves the second unchanged. This gives \eqref{Jfrac}, because
$a_{2n}+a_{2n+1}-1=32n^2+32n+11$ and $a_{2n-1}a_{2n}=16n^2(16n^2-1)$.

(4) For a sequence with generating function \eqref{Jfrac} the Hankel determinant of order $r+1$
is $\prod_{k=1}^{r}\bigl(16k^2(16k^2-1)\bigr)^{r+1-k}$ \cite[Theorem~11]{Kra}. The same holds for
$\sigma_n$, whose Jacobi continued fraction has the same coefficients of $z^2$. Finally
$\det\bigl((-1)^{i+j}S_{2i+2j+1}\bigr)=\det\bigl(S_{2i+2j+1}\bigr)$, and the limit follows from (1).
\end{proof}

\begin{remark}\label{cf-history}
The series evaluation in Theorem~\ref{principal}(1) is equivalent to \cite[(5.2)]{Sok20}. The
Jacobi continued fraction \eqref{Jfrac} is \cite[(5.11)]{Sok20}; it is a special case of a continued fraction
of Stieltjes, see \cite{Sok20}. Continued fractions equivalent to \eqref{Sfrac} and \eqref{Jfrac}
are recorded in \cite[A000464]{OEIS}, where they are due to S.~N.~Gladkovskii, and P.~Bala remarks
there that the binomial transform of the sequence $S_{2n+1}$ has a Stieltjes continued fraction. For related
continued fractions see Dumont \cite{Dum}. What Theorem~\ref{principal} adds
is the meaning of these numbers for $D$: the binomial transform $\sigma_n$, which has the
simplest continued fraction, is exactly the sequence of expansion coefficients of $D$ at $q=1$,
and the Hankel determinants are those of the limits of $W_n$. The proof of (3) through the
continuous dual Hahn polynomials keeps the paper self-contained.
\end{remark}

\begin{remark}\label{positivity-remark}
For $-1<q<0$ all terms of $N_n(q)$ are positive, and $W_n(q)$ is the $n$th moment of the random
variable $(2K+1)^2$, where $K$ takes the value $k$ with probability $(2k+1)|q|^{T_k}/D(q)$. By
Theorem~\ref{oddorder}, as $q=-e^{-t}\to-1$ the moments of $\frac t8(2K+1)^2$ tend to $n!$, the
moments of the exponential distribution. At the other end, by Theorem~\ref{principal} and
Proposition~\ref{moments}, the limits at $q=1$ are the moments of $-4Y^2$, where $Y$ has the
distribution $\mu$.
\end{remark}

\section{Expansions in the unit disc}\label{sec:disc}

We begin with a form of the Rogers--Fine identity that is adapted to the signs $\eps_k$.

\begin{lemma}\label{RF}
For $|q|<1$ and $|z|<1$,
\begin{equation}\label{RFid}
\sum_{k\ge0}\eps_kz^kq^{T_k}=(1-z)\sum_{n\ge0}\frac{(zq^2;q^2)_n}{(-zq;q^2)_{n+1}}\,z^n .
\end{equation}
Both sides are power series in $z$ whose coefficients are polynomials in $q$ with integer
coefficients, and \eqref{RFid} holds as an identity between such series. Hence $q$ may be any
complex number.
\end{lemma}

\begin{proof}
The Rogers--Fine identity \cite[(14.1)]{Fine}, \cite[Chapter~9]{AB1}, with $aq$ and $bq$
renamed $a$ and $b$, states that for $|q|<1$ and $|\tau|<1$
\[
\sum_{n\ge0}\frac{(a;q)_n}{(b;q)_n}\,\tau^n
=\sum_{n\ge0}\frac{(a;q)_n\,(a\tau q/b;q)_n\,b^n\tau^nq^{n^2-n}\,(1-a\tau q^{2n})}{(b;q)_n\,(\tau;q)_{n+1}} .
\]
Replace $q$ by $q^2$ and take $a=zq^2$, $b=-zq^3$ and $\tau=z$. Then $a\tau q^2/b=-zq$ and
\[
\frac{(zq^2;q^2)_n}{(z;q^2)_{n+1}}=\frac{1}{1-z},\qquad
\frac{(-zq;q^2)_n}{(-zq^3;q^2)_n}=\frac{1+zq}{1+zq^{2n+1}} .
\]
The right side becomes
\[
\frac{1+zq}{1-z}\sum_{n\ge0}(-1)^nz^{2n}q^{2n^2+n}\bigl(1-zq^{2n+1}\bigr)
=\frac{1+zq}{1-z}\sum_{k\ge0}\eps_kz^kq^{T_k},
\]
because $T_{2n}=2n^2+n$, $T_{2n+1}=2n^2+3n+1$, $\eps_{2n}=(-1)^n$ and $\eps_{2n+1}=(-1)^{n+1}$. The
left side is $(1+zq)\sum_n(zq^2;q^2)_nz^n/(-zq;q^2)_{n+1}$. This proves \eqref{RFid}. On the right
of \eqref{RFid} the coefficient of $z^k$ involves only the terms with $n\le k$ and is a
polynomial in $q$ with integer coefficients. It agrees with $\eps_kq^{T_k}$ for $|q|<1$, hence
identically.
\end{proof}

Let
\[
t_n(q)=\frac{(q^2;q^2)_n}{(-q;q^2)_{n+1}}\qquad(n\ge0).
\]
The next theorem is, after replacing $q$ by $-q$, the sum of tails identity for $\psi(q)$ of
Andrews, Jim\'enez-Urroz and Ono \cite{AJO}; compare \cite[Theorem~2.1]{RS}, where it is quoted
in another normalisation. We include a short
proof, which also yields the product in \eqref{psi}.

\begin{theorem}\label{tails}
For $|q|<1$ we have $\lim_{n\to\infty}t_n(q)=\psi(-q)$ and
\[
D(q)=\psi(-q)\Bigl(1-2\sum_{j\ge1}\frac{q^{2j}}{1-q^{2j}}-2\sum_{j\ge0}\frac{q^{2j+1}}{1+q^{2j+1}}\Bigr)
+2\sum_{n\ge0}\bigl(\psi(-q)-t_n(q)\bigr).
\]
\end{theorem}

\begin{proof}
Let
\[
g_n(z)=\frac{(zq^2;q^2)_n}{(-zq;q^2)_{n+1}},\qquad g_\infty(z)=\frac{(zq^2;q^2)_\infty}{(-zq;q^2)_\infty},
\]
and fix $\rho$ with $1<\rho<|q|^{-1}$; we may assume $q\ne0$, since for $q=0$ both sides of the
asserted formula are equal to $1$. On the disc $|z|\le\rho$ these functions are analytic, and
$g_n(z)/g_\infty(z)=(-zq^{2n+3};q^2)_\infty/(zq^{2n+2};q^2)_\infty=1+O(|q|^{2n})$ uniformly. Hence
$H(z)=\sum_{n\ge0}\bigl(g_n(z)-g_\infty(z)\bigr)z^n$ is analytic for $|z|<\rho$. For $|z|<1$,
Lemma~\ref{RF} gives
\[
\Phi(z):=\sum_{k\ge0}\eps_kz^kq^{T_k}=(1-z)\sum_{n\ge0}g_n(z)z^n=g_\infty(z)+(1-z)H(z),
\]
and since $\Phi$ is entire, this holds for $|z|<\rho$. At $z=1$ we get
$\psi(-q)=\Phi(1)=g_\infty(1)=(q^2;q^2)_\infty/(-q;q^2)_\infty$, which is the limit of
$t_n(q)=g_n(1)$ and equals the product in \eqref{psi}, because
$(-q;q^2)_\infty(q;q^2)_\infty=(q^2;q^4)_\infty$. Next, $D=\Phi(1)+2\Phi'(1)$ and
$\Phi'(1)=g_\infty'(1)-H(1)$, where $H(1)=\sum_n\bigl(t_n(q)-\psi(-q)\bigr)$ and
\[
\frac{g_\infty'(1)}{g_\infty(1)}=-\sum_{j\ge1}\frac{q^{2j}}{1-q^{2j}}-\sum_{j\ge0}\frac{q^{2j+1}}{1+q^{2j+1}} .
\]
This is the assertion.
\end{proof}

Since $\psi(-e^{-t})$ vanishes to infinite order as $t\to0^+$ by \eqref{intro-theta}, the theorem
formally turns into $D(e^{-t})=-2\sum_nt_n(e^{-t})$. In view of Theorem~\ref{radial}(4) and
Remark~\ref{Lvalues} this is the identity \eqref{AJO}.

We now define the integers $c_N$. For $N\ge1$ let $c_N'$ be the number of divisors of $N$ if $N$
is odd, and the number of divisors of $N$ that are multiples of $4$ if $N$ is even. For an even
integer $m\ge2$ let $\mathcal P_m$ be the set of nonempty partitions in which every part is at
least $m$ and every part congruent to $2$ modulo $4$ is less than $2m$. Let $c_N''$ be the number
of pairs $(m,\pi)$ in which $m\ge2$ is even and $\pi\in\mathcal P_m$ is a partition of $N$;
necessarily $m\le N$, so this number is finite.

\begin{theorem}\label{positive}
For $|q|<1$,
\[
\frac{D(q)}{\psi(-q)}=1-2\sum_{N\ge1}c_Nq^N,\qquad c_N=c_N'+c_N'',
\]
and $c_N\ge1$ for every $N\ge1$.
\end{theorem}

\begin{proof}
First, the factor in parentheses in Theorem~\ref{tails} is $1-2\sum_{N\ge1}c_N'q^N$. Indeed, the
coefficient of $q^N$ in $\sum_{j\ge1}q^{2j}/(1-q^{2j})$ is the number of even divisors of $N$. In
$\sum_{j\ge0}q^{2j+1}/(1+q^{2j+1})$ it is the number of divisors of $N$ if $N$ is odd, and minus the
number of odd divisors of $N$ if $N$ is even. For even $N$ the even divisors not divisible by
$4$ are in bijection with the odd divisors.

Second, let $m=2n+2$. Using $(-x;q^2)_\infty=(x^2;q^4)_\infty/(x;q^2)_\infty$ we find
\[
\frac{t_n(q)}{\psi(-q)}=\frac{(-q^{2n+3};q^2)_\infty}{(q^{2n+2};q^2)_\infty}
=\frac{(q^{2m+2};q^4)_\infty}{(q^{m};q)_\infty}
=\prod_{\substack{j\ge m\\ j\notin\{2m+2,\,2m+6,\,2m+10,\dots\}}}\frac{1}{1-q^j}.
\]
Since $2m\equiv0\pmod4$, the excluded integers are exactly the $j\equiv2\pmod4$ with $j\ge2m$. So
$t_n/\psi(-q)=1+\sum_{\pi\in\mathcal P_m}q^{|\pi|}$, where $|\pi|$ is the sum of the parts of
$\pi$, and Theorem~\ref{tails} gives the expansion.

Finally, if $N$ is odd then $c_N'\ge1$. If $N$ is even, then $m=N$ and the partition of $N$ with
the single part $N$ form a pair counted by $c_N''$.
\end{proof}

\begin{example}
The first values are
\[
(c_N)_{N\ge1}=1,\,1,\,3,\,4,\,5,\,5,\,8,\,12,\,14,\,14,\,20,\,28,\,31,\,35,\,49,\,62,\dots .
\]
For $N=6$ we have $c_6'=0$. For $m=2$ the partitions of $6$ in $\mathcal P_2$ are $4+2$, $3+3$ and
$2+2+2$; the partition with the single part $6$ is excluded, because $6\equiv2\pmod4$ and $6\ge2m$. For $m=4$
and for $m=6$ the only partition of $6$ in $\mathcal P_m$ is the one with the single part $6$. So $c_6=5$.
\end{example}

Let $\podf(N)$ be the number of partitions of $N$ in which the odd parts are distinct
\cite{HS}, with $\podf(0)=1$ and $\podf(N)=0$ for $N<0$. The product \eqref{psi} and
Theorem~\ref{positive} give a pair of recurrences for this function.

\begin{corollary}\label{podcor}
For every $N\ge1$,
\begin{align*}
\podf(N)-\podf(N-1)-\podf(N-3)+\podf(N-6)+\podf(N-10)-\cdots&=0,\\
\podf(N)-3\podf(N-1)-5\podf(N-3)+7\podf(N-6)+9\podf(N-10)-\cdots&=-2c_N .
\end{align*}
In particular the left side of the second line is negative.
\end{corollary}

\begin{proof}
The generating function of $\podf(N)$ is $(-q;q^2)_\infty/(q^2;q^2)_\infty=1/\psi(-q)$. The two
lines are the coefficients of $q^N$ in $\psi(-q)\cdot\psi(-q)^{-1}=1$ and in the identity
$D(q)\cdot\psi(-q)^{-1}=1-2\sum_Nc_Nq^N$ of Theorem~\ref{positive}.
\end{proof}

The first recurrence is the analogue of Euler's recurrence for the partition function.
The second has the shape of the classical recurrence that expresses the sum of the divisors of
$N$ through the partition function, with $2c_N$ in the role of the divisor sum.

For the function $D(-q)$ there is a second expansion, of a different shape.

\begin{proposition}\label{warnaar}
For $|q|<1$,
\[
\sum_{k\ge0}(2k+1)q^{T_k}=\psi(q)\Bigl\{1+4\sum_{j\ge1}\frac{q^j}{1+q^j}
-2\sum_{n\ge1}\frac{(q;q^2)_n}{(q^2;q^2)_n}\,\frac{q^n}{1-q^n}\Bigr\}.
\]
The left side is $D(-q)$.
\end{proposition}

\begin{proof}
Write $(x_1,\dots,x_s;q)_n=\prod_{i}(x_i;q)_n$. Warnaar's extension of the Jacobi triple product
identity \cite[Theorem~1.5]{War}, for which another proof is given in \cite{AW}, states that
\[
1+\sum_{n\ge1}(-1)^nq^{n(n-1)/2}(a^n+b^n)=(q,a,b;q)_\infty\sum_{n\ge0}\frac{(ab/q;q)_{2n}\,q^n}{(q,a,b,ab;q)_n} .
\]
Apply $b\,\frac{\partial}{\partial b}$ and then put $ab=q$. On the right, $(ab/q;q)_{2n}$ vanishes
at $ab=q$ for $n\ge1$, and $b\,\frac{\partial}{\partial b}(ab/q;q)_{2n}=-(q;q)_{2n-1}$ there. Hence,
for $ab=q$,
\[
\sum_{n\ge1}(-1)^n\,n\,b^nq^{n(n-1)/2}=(q,a,b;q)_\infty\Bigl\{-\sum_{j\ge0}\frac{bq^j}{1-bq^j}
-\sum_{n\ge1}\frac{(q;q)_{2n-1}\,q^n}{(q;q)_n^2\,(a;q)_n\,(b;q)_n}\Bigr\}.
\]
Take $b=-1$ and $a=-q$. The left side becomes $\sum_{k\ge0}(k+1)q^{T_k}$, and
$(q,-q,-1;q)_\infty=2(q^2;q^2)_\infty(-q;q)_\infty=2(q^2;q^2)_\infty/(q;q^2)_\infty=2\psi(q)$ by
Theorem~\ref{tails}, applied with $-q$ in place of $q$. Using $(-1;q)_n=2(-q;q)_{n-1}$ one finds
\[
\frac{(q;q)_{2n-1}}{(q;q)_n^2\,(-q;q)_n\,(-1;q)_n}=\frac{(q;q^2)_n}{2\,(q^2;q^2)_n\,(1-q^n)} .
\]
The proposition follows, because $\sum_k(2k+1)q^{T_k}=2\sum_k(k+1)q^{T_k}-\psi(q)$.
\end{proof}

\section{The infinite product expansion and the zero of least modulus}\label{sec:product}

Every power series with integer coefficients and constant term $1$ can be written in exactly
one way as a product $\prod_{r\ge1}(1-q^r)^{\kappa_r}$ with integers $\kappa_r$. We determine the
exponents $\kappa_r$ for the series $D$. Consider an alphabet that has $2c_j$ letters of
weight $j$ for every $j\ge1$, with some total order. The weight of a word is the sum of the
weights of its letters; there are only finitely many words of a given weight. A Lyndon word is a nonempty word that is strictly smaller, in the
lexicographic order, than each of its proper nonempty suffixes. Let $\ell_r$ be the number of Lyndon
words of weight $r$, and let
\[
C(q)=2\sum_{N\ge1}c_Nq^N .
\]

\begin{theorem}\label{product}
The exponents in the product expansion $D(q)=\prod_{r\ge1}(1-q^r)^{\kappa_r}$ are
\[
\kappa_r=\ell_r+1\quad\text{if } r\not\equiv2\pmod4,\qquad \kappa_r=\ell_r\quad\text{if } r\equiv2\pmod4.
\]
In particular $\kappa_r\ge\ell_r\ge2c_r\ge2$ for every $r\ge1$.
\end{theorem}

\begin{proof}
The generating function of all words by weight is $1/(1-C(q))$. By the theorem of Chen, Fox and
Lyndon \cite{CFL}, \cite[Theorem~5.1.5]{Lot}, every word is in exactly one way a product of a
nonincreasing sequence of Lyndon words. Therefore $1/(1-C(q))=\prod_{r\ge1}(1-q^r)^{-\ell_r}$. By
Theorem~\ref{positive} and \eqref{psi},
\[
D(q)=\psi(-q)\bigl(1-C(q)\bigr)=\prod_{r\ \mathrm{odd}}(1-q^r)\prod_{4\mid r}(1-q^r)\prod_{r\ge1}(1-q^r)^{\ell_r}.
\]
Single letters are Lyndon words, so $\ell_r\ge2c_r$, and $c_r\ge1$ by Theorem~\ref{positive}.
\end{proof}

\begin{example}
The first values are
\begin{align*}
(\ell_r)_{r\ge1}&=2,\,3,\,12,\,32,\,92,\,254,\,770,\,2291,\dots,\\
(\kappa_r)_{r\ge1}&=3,\,3,\,13,\,33,\,93,\,254,\,771,\,2292,\dots .
\end{align*}
For instance $\ell_2=3$: there are $2c_2=2$ letters of weight $2$, and the word
$\mathsf a\mathsf b$ formed by the two letters $\mathsf a<\mathsf b$ of weight $1$.
\end{example}

\begin{corollary}\label{signs}
Let $\kappa_r$ be as in Theorem~\ref{product}.
\begin{enumerate}
\item All coefficients of $1/D(q)=1+3q+9q^2+32q^3+111q^4+\cdots$ are positive. More precisely,
$1/D$ is the generating function for partitions in which the parts of size $r$ come in $\kappa_r$
colours.
\item $W_1(q)=1-8\sum_{m\ge1}w_mq^m$ with $w_m=\sum_{d\mid m}d\,\kappa_d$. Every $w_m$ is a positive
integer divisible by $3$, and $w_m\ge2mc_m$. The first values are $w_1=3$, $w_2=9$, $w_3=42$,
$w_4=141$ and $w_5=468$.
\end{enumerate}
\end{corollary}

\begin{proof}
(1) is immediate from Theorem~\ref{product}. For (2), $W_1=1+\Theta D/D=1+8qD'/D$ by
\eqref{theta-op}, and $-qD'/D=\sum_{r}r\kappa_rq^r/(1-q^r)=\sum_mw_mq^m$. From
$D=\psi(-q)(1-C)$ and \eqref{psi} we see that $-qD'/D$ is the sum of the series
$\sum rq^r/(1-q^r)$ over all $r\not\equiv2\pmod4$, which has nonnegative coefficients, and of
$qC'/(1-C)$, whose coefficient of $q^m$ is at least that of $qC'$, namely $2mc_m$. Finally
$j^3-j$ is divisible by $24$ for odd $j$, so $N_1-D$ and $W_1-1=(N_1-D)/D$ have coefficients
divisible by $24$, and $3\mid w_m$.
\end{proof}

\begin{theorem}\label{zero}
There is a real number $q_0$ with $\frac14<q_0<\frac13$ such that
\begin{enumerate}
\item $q_0$ is a simple zero of $D$,
\item $q_0$ is the only zero of $D$ in the closed disc $|q|\le \frac12$,
\item $q_0$ is the only zero of $D$ in the real interval $(-1,1)$.
\end{enumerate}
\end{theorem}

\begin{proof}
We use the cubic comparison in \cite[Theorem~7.1]{RaghavaPreprint}.
Put $P(q)=1-3q-5q^3$. For $q=e^{i\theta}/2$ and $x=\cos\theta$, we have
\[
 |P(q)|^2=\frac{81}{64}+\frac{1-x}{4}(20x^2+5x+2)\ge\frac{81}{64},
\]
since the quadratic factor is positive for real $x$. Throughout $|q|\le1/2$, the terms
of $D-P$ are bounded in absolute value by $u_k=(2k+1)2^{-T_k}$, $k\ge3$. The ratio
$u_{k+1}/u_k=\frac{2k+3}{2k+1}2^{-(k+1)}$ is at most $9/112$. Hence
\[
 |D(q)-P(q)|\le\frac{7/64}{1-9/112}=\frac{49}{412}<\frac18 .
\]
On $|q|=1/2$ this is strictly less than $9/8\le|P(q)|$. Rouch\'e's theorem gives the
same number of zeros for $D$ and $P$ in $|q|<1/2$, counted with multiplicity, and the
strict inequality excludes zeros of $D$ on the boundary.
The cubic $5q^3+3q-1$ is strictly increasing on the real line and changes sign between $0$
and $1/2$. If $\xi$ is its real root and $\eta,\bar\eta$ are the other two roots, then
$|\eta|^2=1/(5\xi)>2/5>1/4$. Thus $D$ has exactly one zero in the disc, and it is simple.
Since $P(1/4)=11/64>1/8$ and $P(1/3)=-5/27<-1/8$, the same tail bound gives
$D(1/4)>0>D(1/3)$. The zero $q_0$ therefore lies in $(1/4,1/3)$.

Finally, by \eqref{psi} and Theorem~\ref{positive}, the zeros of $D$ in $(0,1)$
are the solutions of $C(q)=1$. All coefficients $2c_N$ of $C$ are positive, so $C$
is strictly increasing there. This gives uniqueness on $(0,1)$. For $-1<q<0$
every term of $D(q)$ is positive, and $D(0)=1$.
\end{proof}

\begin{corollary}\label{growth}
There is a number $q_1$ with $\frac12<q_1<1$ such that, as $m\to\infty$,
\[
w_m=q_0^{-m}+O(q_1^{-m}),\qquad [q^m]\frac{1}{D(q)}=\frac{q_0^{-m}}{-q_0D'(q_0)}+O(q_1^{-m}),\qquad
\kappa_m\sim\frac{q_0^{-m}}{m}.
\]
In particular $W_1$ has a simple pole at $q_0$, its Taylor series at $0$ has radius of
convergence exactly $q_0$, and $[q^m]W_1\sim-8\,q_0^{-m}$.
\end{corollary}

\begin{proof}
By Theorem~\ref{zero}, $D$ has no zero on the circle $|q|=1/2$. Compactness gives a zero-free
annulus about this circle, so there is $q_1$ with $\frac12<q_1<1$ such that
$q_0$ is the only zero of $D$ in $|q|\le q_1$. In this disc $1/D$ and $qD'/D$ are analytic except
for a simple pole at $q_0$, with principal parts $1/\bigl(D'(q_0)(q-q_0)\bigr)$ and $q_0/(q-q_0)$.
Subtracting the principal parts and applying Cauchy's estimate on $|q|=q_1$ gives the first two
statements; recall that $qD'/D=-\sum_mw_mq^m$. The third follows by M\"obius inversion of
$w_m=\sum_{d\mid m}d\,\kappa_d$: we get $m\kappa_m=w_m+O(m\,q_0^{-m/2})$, because the other terms involve only $w_d$ with $d\le m/2$.
The statements about $W_1=1+8qD'/D$ are now clear.
\end{proof}

\begin{remark}\label{oeis-remark}
The coefficients of $1/D$ form sequence A202143 of \cite{OEIS}, and the numbers
$w_m/3$ form sequence A202144. Both entries are due to P.~D.~Hanna.
Corollary~\ref{growth} identifies the common radius of convergence as $q_0$
and gives the limits
\[
 \lim_{m\to\infty}q_0^m[q^m]\frac1{D(q)}=-\frac1{q_0D'(q_0)},\qquad
 \lim_{m\to\infty}q_0^m\,\frac{w_m}{3}=\frac13 .
\]
\end{remark}

Zeros of partial theta functions as functions of the other variable have been studied in detail;
see Sokal \cite{Sok12}. Here the positivity of the $c_N$ gives the zero of least modulus at once.

\section{Terminating identities at roots of unity}\label{sec:term}

By Lemma~\ref{RF} we may specialise $q$ in the Rogers--Fine identity \eqref{RFid} to a root of
unity. The series then becomes periodic and can be summed.

\begin{theorem}\label{terminating}
Let $\lambda$ be a root of unity of order $h$. If $h$ is even, then
\begin{equation}\label{term-even}
\sum_{k=0}^{h-1}\lambda^{T_k}z^k=(1-z)\bigl(1+z^{h/2}\bigr)\sum_{r=0}^{h/2-1}
\frac{(z\lambda^2;\lambda^2)_r}{(z\lambda;\lambda^2)_{r+1}}\,z^r ,
\end{equation}
and if $h$ is odd, then
\begin{equation}\label{term-odd}
\sum_{k=0}^{h-1}\lambda^{T_k}z^k=(1-z)\sum_{r=0}^{h-1}\frac{(z\lambda^2;\lambda^2)_r}{(z\lambda;\lambda^2)_{r+1}}\,z^r .
\end{equation}
Both are identities of rational functions of $z$.
\end{theorem}

\begin{proof}
Put $q=-\lambda$ in \eqref{RFid}. Since $\eps_k(-\lambda)^{T_k}=\lambda^{T_k}$, we get the identity
of formal power series
\[
\sum_{k\ge0}\lambda^{T_k}z^k=(1-z)\sum_{n\ge0}g_nz^n,\qquad g_n=\frac{(z\lambda^2;\lambda^2)_n}{(z\lambda;\lambda^2)_{n+1}} .
\]
Let $d$ be the order of $\lambda^2$, so $d=h/2$ if $h$ is even and $d=h$ if $h$ is odd. Since
$\prod_{\eta^d=1}(1-x\eta)=1-x^d$, we have
$(z\lambda^2;\lambda^2)_{n+d}=(1-z^d)(z\lambda^2;\lambda^2)_n$ and
$(z\lambda;\lambda^2)_{n+1+d}=(1-\lambda^dz^d)(z\lambda;\lambda^2)_{n+1}$. Hence
$g_{n+d}=g_n\,(1-z^d)/(1-\lambda^dz^d)$, and summing a geometric series we find
\[
\sum_{n\ge0}g_nz^n=\frac{1-\lambda^dz^d}{1-\lambda^dz^d-z^d+z^{2d}}\;\sum_{r=0}^{d-1}g_rz^r .
\]
If $h$ is even, then $\lambda^d=-1$ and the fraction is $(1+z^{h/2})/(1+z^{h})$, while the left
side of the first display is $\sum_{k<h}\lambda^{T_k}z^k/(1+z^h)$ by \eqref{anti}. If $h$ is odd,
then $\lambda^d=1$ and the fraction is $1/(1-z^h)$, while the left side is
$\sum_{k<h}\lambda^{T_k}z^k/(1-z^h)$, because $\lambda^{T_{k+h}}=\lambda^{T_k}$.
\end{proof}

For example, for $\lambda=i$ the identity \eqref{term-even} reads
\[
1+iz-iz^2-z^3=(1-z)(1+z^2)\Bigl\{\frac{1}{1-iz}+\frac{z\,(1+z)}{1+z^2}\Bigr\}.
\]

\begin{corollary}\label{finite-sum}
Let $\zeta$ be a root of unity such that $\lambda=-\zeta$ has even order $h$. Then
\begin{equation}\label{finite-val}
\lim_{t\to0^+}D(\zeta e^{-t})=\sum_{k=0}^{h-1}k\,\lambda^{T_k}
=-2\sum_{r=0}^{h/2-1}\frac{(\zeta^2;\zeta^2)_r}{(-\zeta;\zeta^2)_{r+1}} .
\end{equation}
More generally, for $|x|<\pi/2h$,
\begin{equation}\label{finite-gen}
\sum_{n\ge0}L_n(\lambda)\,\frac{x^{2n+1}}{(2n+1)!}
=-\frac{e^{-x}(1-e^{-2x})(1+e^{-hx})}{1+e^{-2hx}}\sum_{r=0}^{h/2-1}
\frac{(e^{-2x}\lambda^2;\lambda^2)_r}{(e^{-2x}\lambda;\lambda^2)_{r+1}}\,e^{-2rx}.
\end{equation}
\end{corollary}

\begin{proof}
Let $P(z)$ be the left side of \eqref{term-even}. By the proof of Lemma~\ref{pf}(1) we have
$\B_\lambda(x)=e^{-x}P(e^{-2x})/(1+e^{-2hx})$. Together with \eqref{Lgen} and \eqref{term-even}
this gives \eqref{finite-gen}. The first equality in \eqref{finite-val} is contained in parts (1) and (2) of
Theorem~\ref{radial}.
For the second, note that the denominators $1-z\lambda^{2j+1}$ in \eqref{term-even} do not vanish
at $z=1$, since $h$ is even. So we may differentiate \eqref{term-even} at $z=1$, which gives
\[
\sum_{k=0}^{h-1}k\,\lambda^{T_k}=P'(1)=-2\sum_{r=0}^{h/2-1}\frac{(\lambda^2;\lambda^2)_r}{(\lambda;\lambda^2)_{r+1}},
\]
and $\lambda=-\zeta$.
\end{proof}

The right side of \eqref{finite-val} is $-2\sum_{n\ge0}t_n(\zeta)$: the terms with $n\ge h/2$
vanish, because $\zeta^2$ has order $h/2$. For $|q|<1$ the series $\sum_nt_n(q)$ diverges, since
its terms tend to $\psi(-q)\ne0$. So \eqref{finite-val} is an exact form, at each root of unity
where it makes sense, of the formal relation $D=-2\sum_nt_n$ noted after Theorem~\ref{tails}. For
$\zeta=1$ it is the constant term of \eqref{AJO}. A series equivalent to $\sum_nt_n(q)$ is one
component of the vector-valued quantum modular form of Rolen and Schneider \cite{RS}. They show
that, at the roots of unity where it is defined, it agrees with a half-derivative of a theta
function \cite[Corollary~2.1]{RS}; in our notation this half-derivative is $D$. So the equality of
the outer members of \eqref{finite-val} is, up to normalisation, contained in their work. Finite evaluations of the same nature for mock theta functions are in \cite{FOR}. What
Theorem~\ref{terminating} adds is the polynomial identity in $z$, which gives an elementary
proof and the complete generating function \eqref{finite-gen}.

\begin{example}
For $\zeta=-i$ the right side of \eqref{finite-val} is
$-2\bigl\{\frac{1}{1-i}+\frac{2}{(1-i)(1+i)}\bigr\}=-3-i$, in agreement with
Example~\ref{exroots}(a). For $\zeta=e^{2\pi i/3}$ it has three terms and equals
$\frac12(-11+3\sqrt3\,i)$.
\end{example}

\begin{remark}
For odd $h$ the left side of \eqref{term-odd} at $z=1$ is the Gauss sum $G_\lambda$ of
Theorem~\ref{transf-odd}. On the right the terms with $r\ge(h-1)/2$ have a simple pole at $z=1$,
which is cancelled by the factor $1-z$, and one obtains
\[
G_\lambda=\sum_{r=(h-1)/2}^{h-1}\frac{(\lambda^2;\lambda^2)_r}{\prod_{0\le j\le r,\ 2j+1\ne h}(1-\lambda^{2j+1})}.
\]
\end{remark}

\section{Concluding remarks}\label{sec:final}

It remains open whether $D$ has infinitely many zeros in the unit disc.
For the higher moments the question is whether $N_n(q_0)\ne0$ for every $n\ge2$.
The case $n=1$ follows from $N_1(q_0)=8q_0D'(q_0)$. An affirmative answer would give
\[
 [q^m]W_n(q)\sim-\frac{N_n(q_0)}{q_0D'(q_0)}\,q_0^{-m}
 \qquad(m\to\infty)
\]
for every fixed $n\ge1$. We do not know a combinatorial proof of the second
recurrence in Corollary~\ref{podcor} in which the numbers $c_N$ arise from a direct
count. Finally, the limits of $W_n$ at $q=1$ form a moment sequence with an
explicit Jacobi continued fraction. Do the limits $L_n(\lambda)/L_0(\lambda)$ at
other roots of unity have a comparable structure?

\section*{Statements and declarations}
\noindent\textbf{Funding.} The author received no external funding for this work.

\smallskip\noindent\textbf{Competing interests.} The author declares no competing interests.

\end{document}